\documentclass[10pt, twoside]{amsart}

\usepackage[english]{babel}

\usepackage{amsmath,amssymb}
\newcommand{\Z}{{\mathbb Z}}

\newcommand{\F}{{\mathbb F}}
\newcommand{\N}{{\mathbb N}}
\newcommand{\C}{{\mathbb C}}

\newcommand{\G}{{\mathbb G}}

\newcommand{\p}{{\mathfrak p}}

\newcommand{\m}{{\mathfrak m}}
\newcommand{\X}{{\mathfrak X}}

\newcommand{\Var}{\mathcal V}

\newcommand{\gammasep}{{\gamma}_{\operatorname{sep}}}
\newcommand{\Varsep}{{\mathcal V}_{\operatorname{sep}}}
\newcommand{\Isep}{{\mathcal I}_{\operatorname{sep}}}
\newcommand{\upto}{,\ldots ,}
\newcommand{\mmax}{\mathfrak{m}}
\newcommand{\nmax}{\mathfrak{n}}
\newcommand{\pprime}{\mathfrak{p}}
\newcommand{\qprime}{\mathfrak{q}}
\DeclareMathOperator{\codim}{codim}

\DeclareMathOperator{\GL}{GL}
\DeclareMathOperator{\id}{id}
\DeclareMathOperator{\cmdef}{cmd}

\DeclareMathOperator{\cidef}{cid}

\DeclareMathOperator{\cid}{cid}

\DeclareMathOperator{\Spec}{Spec}
\DeclareMathOperator{\Quot}{Quot}

\DeclareMathOperator{\depth}{depth}
\DeclareMathOperator{\height}{ht}
\DeclareMathOperator{\trdeg}{trdeg}

\DeclareMathOperator{\characteristic}{char}

\newtheorem{thm}{Theorem}[section]
\newtheorem{prop}[thm]{Proposition}
\newtheorem{lem}[thm]{Lemma}
\newtheorem{cor}[thm]{Corollary}
\newtheorem{conj}[thm]{Conjecture}
\theoremstyle{definition}
\newtheorem{definition}[thm]{Definition}
\newtheorem{example}[thm]{Example}
\theoremstyle{remark}
\newtheorem{remark}[thm]{Remark}
\theoremstyle{theorem}

\numberwithin{equation}{section}

\usepackage{stmaryrd}

\usepackage[arrow, matrix, curve]{xy}

\usepackage{enumerate}

\usepackage[pdfpagemode=None,
			baseurl={https://www-m11.ma.tum.de},
           	pdftitle={Reimers 2017},
           	pdfauthor={Fabian Reimers}]{hyperref}

\begin{document}

\title{Separating Invariants of Finite Groups}

\author{Fabian Reimers}

\address{Technische Universi\"at M\"unchen, Zentrum Mathematik - M11, 
Boltzmannstr.~3, 85748 Garching, Germany}

\email{reimers@ma.tum.de}

\date{March 14, 2017}



\begin{abstract}
This paper studies separating invariants of finite groups acting on affine varieties through automorphisms. 
Several results, proved by Serre, Dufresne, Kac-Watanabe and Gordeev, and Jeffries and Dufresne exist that relate properties of the invariant ring or a separating subalgebra to properties of the group action. All these results are limited to the case of linear actions on vector spaces. The goal of this paper is to lift this restriction by extending these results to the case of (possibly) non-linear actions on affine varieties.

Under mild assumptions on the variety and the group action, we prove that polynomial separating algebras can exist only for reflection groups. 
The benefit of this gain in generality is demonstrated by an application to the semigroup problem in multiplicative invariant theory.

Then we show that separating algebras which are complete intersections in a certain codimension can exist only for 2-reflection groups. Finally we prove that a separating set of size $n + k - 1$ (where $n$ is the dimension of $X$) can exist only for $k$-reflection groups.

Several examples show that most of the assumptions on the group action and the variety that we make cannot be dropped.
\end{abstract}

\maketitle

\section*{Introduction}

Invariant theory studies the ring of those regular functions on an affine variety $X$ that are constant on the orbits of an action of a linear algebraic group $G$ on $X$, where the action is given by a morphism $G \times X \to X$. This paper considers the case where $G$ is a finite group. We will always assume that the ground field $K$, over which $G$ and $X$ are defined, is algebraically closed, and refer to this setting by calling $X$ a $G$-variety. 

The group action on $X$ induces an action on the coordinate ring $K[X]$ via $\sigma \cdot f := f \circ \sigma^{-1}$ for $\sigma \in G$ and $f\in K[X]$. The elements fixed by this action 
are called invariants and they form a subalgebra of $K[X]$: 
\[ K[X]^G := \{ f \in K[X] \mid \sigma \cdot f = f \text{ for all }\sigma \in G\},\]
which is called the invariant ring.
Although $K[X]^G$ is finitely generated as a $K$-algebra by a classical theorem of Noether \cite{noether1926endlichkeitssatz}, its minimal number of generators can be very large even for small groups and low-dimensional representations $X = V$ of $G$ (see e.g. \cite[Table in Section 5] {kamke2012algorithmic}). Derksen and Kemper \cite{derksen2002computational} introduced the notion of separating invariants as a weaker concept than generating invariants. A subset $S \subseteq K[X]^G$ of the invariant ring is called \emph{separating} if for all $x,\, y \in X$ the following holds: If there exists an invariant $f \in K[X]^G$ with $f(x) \neq f(y)$, then there exists an element $g \in S$ with $g(x) \neq g(y)$. Let $\gammasep$ denote the smallest integer $m$ such that there exists a separating subset of size $m$. It has been known for a while that $\gammasep$ is bounded above by $2n+1$ where $n$ is the transcendence degree of the invariant ring (see \cite{dufresne2008thesis} and \cite{kamke2012algorithmic}). In the case of finite groups $n$ is just the dimension of the variety $X$.

Therefore, it can make sense to shift the focus on separating sets of invariants rather than possibly much larger and more complicated sets of generating invariants (see \cite{kemper2009separating} for other aspects in which separating invariants are better behaved than generating invariants). 

This naturally leads to the question if separating sets of the smallest possible size and separating algebras with ''good algebraic properties'' exist for a given group action. We formulate the following two questions explicitly:
\begin{enumerate}
\item[(Q1)] When does there exist a separating algebra $A \subseteq K[X]^G$ that is isomorphic to a polynomial ring (which is equivalent to $\gammasep = n$)?
\item[(Q2)] When does there exist a separating algebra $A \subseteq K[X]^G$ that is a complete intersection (which includes the case $\gammasep = n+1$)? 
\end{enumerate}

In the case of a non-modular linear representation $X = V$ of $G$ the theorem of Shephard and Todd \cite{shephard1954finite}, Chevalley \cite{chevalley1955invariants}, and Serre \cite{serre1968groupes} gives a complete answer for $A = K[V]^G$ to the first question: The invariang ring $K[V]^G$ is a polynomial ring if and only if $G$ is generated by 1-reflections (i.e., by elements that act as identity on a subspace of codimension 1 of $V$). Serre also proved that even in the modular case the invariant ring $K[V]^G$ can only be isomorphic to a polynomial ring if $G$ is a reflection group. Dufresne \cite{dufresne2009separating} generalized this to separating invariants and showed that the theorem of Serre remains true if we replace the invariant ring $K[V]^G$ by any separating subalgebra $A$ of $K[V]^G$. 

The first main result of this article extends this to the more general situation of possibly non-linear actions on varieties that need not be affine spaces. 

\newtheorem*{thm:section2main}{Theorem \ref{TheoremYsep=nImpliesGeneratedByReflections}}
\begin{thm:section2main}
Assume that the $G$-variety $X$ is connected and Cohen-Macaulay and that $G$ is generated by elements having a fixed point. If $\gammasep = n$, then $G$ is generated by reflections.
\end{thm:section2main}

Notice that for a linear action on a vector space $X = V$, every group element has the origin as fixed point. 

A case of non-linear actions to which Theorem~\ref{TheoremYsep=nImpliesGeneratedByReflections} is applicable is given by multiplicative invariant theory. This is a branch of invariant theory which deals with the action of a finite group $G$ on a lattice $L \cong \Z^n$ and the induced action on the group ring $K[L]$. It is studied in detail in Lorenz's book \cite{lorenz2005multiplicative}. Since $K[L]$ is a Laurent polynomial ring in $n$ indeterminates, it can be interpreted as the coordinate ring of an $n$-dimensional algebraic torus $X = \G_m^n$. 

An analogous result in multiplicative invariant theory to the Shephard-Todd-Chevalley-Serre theorem is \cite[Theorem 7.1.1]{lorenz2005multiplicative} which says that under the assumption that $\characteristic(K)$ does not divide $|G|$ the following two statements (among others) are equivalent:
\begin{enumerate}
\item[(a)] $K[L]^G$ is (isomorphic to) a mixed Laurent polynomial ring, i.e., there exists an integer $k \in \{ 0 \upto n \}$ with $K[L]^G \cong K[x_1^{\pm 1} \upto x_k^{\pm 1}, \, x_{k+1}\upto x_n]$,
\item[(b)] $G$ is generated by reflections on $L$ and $K[L]^G$ is a unique factorization domain.
\end{enumerate}

Without any assumptions about the characteristic of $K$, Lorenz \cite{lorenz2001semigroup} proved the following: If $G$ is generated by reflections, then there is a submonoid $M \subseteq K[L]^G$ such that $K[L]^G$ is isormorphic to the semigroup algebra $K[M]$. The question whether the converse of this statement holds is called the ''semigroup problem in multiplicative invariant theory'' (see \cite[Section 1.5]{tesemma2004phd}). Some partial converses are given in \cite{tesemma2004phd} and \cite[Section 10.2]{lorenz2005multiplicative}. We add another one to the list by proving that of the above statements (a) implies (b) independently of the characteristic of $K$.

\newtheorem*{thm:section2mit}{Theorem \ref{TheoremMixedLaurentPolynomialRingInMIT}}
\begin{thm:section2mit}
Let $L$ be a lattice and let $G$ be a finite group acting on $L$ by automorphisms. If $K[L]^G$ is isomorphic to a mixed Laurent polynomial ring, then $G$ is generated by reflections.
\end{thm:section2mit}

\vspace{0.2cm}

For question (Q2) a necessary condition similar to Serre's theorem was found by Kac and Watanabe \cite{kac1982finite} and independently by Gordeev \cite{gordeev1982invariants}. They showed that the invariant ring $K[V]^G$ of a linear representation $V$ of $G$ can only be a complete intersection if $G$ is generated by 2-reflections (i.e., by elements that act as identity on a subspace of codimension 2 of $V$). This was extended by Dufresne \cite{dufresne2009separating} to graded separating subalgebras of $K[V]^G$, and it is now further extended in this article to non-linear actions on varieties, and to separating subalgebras that are not complete intersections globally but satisfy some weaker local property.

\newtheorem*{thm:section3main}{Theorem \ref{TheoremMainTheoremOnCompleteIntersectionAndBireflection}}
\begin{thm:section3main}
Assume that the $G$-variety $X$ is normal and connected and that $X^G \neq \emptyset$. If there exists a finitely generated, separating algebra $A \subseteq K[X]^G$, such that $K[X]^G$ is a finite $A$-module, and that $A$ is a complete intersection in codimension $2 + \cid(A)$ (where $\cid(A)$ is the complete intersection defect of $A$), then $G$ is generated by $2$-reflections. 
\end{thm:section3main}

In particular, the conclusion of the theorem holds if $A$ is a complete intersection. The generalization using the complete intersection defect leads to the following corollary for non-modular (at least 3-dimensional) non-trivial representations $X = V$ of $G$: If $\cid(K[V]^G) \leq n-3$, then $G \setminus \{ \id\}$ must contain an $(n-1)$-reflection (see Theorem~\ref{TheoremUsingIsolatedSingularityForAWeakCorollary}).

\vspace{0.2cm}

Recently, Dufresne and Jeffries \cite{dufresne2013separating} found a remarkable connection, in the case of linear actions on $n$-dimensional affine spaces, between the size of a separating set of invariants and the property of being a $k$-reflection group. They proved that if $\gammasep = n + k - 1$, then $G$ is generated by $k$-reflections. Again we generalize this to affine $G$-varieties, which leads to the following result.

\newtheorem*{thm:section4main}{Theorem \ref{TheoremYsep=n+k-1=>k-reflections}}
\begin{thm:section4main}
Assume that the $G$-variety $X$ is normal and connected and that $G$ is generated by elements having a fixed point in $X$. If $\gammasep = n + k - 1$ (with $k \in \N$), then $G$ is generated by $k$-reflections.
\end{thm:section4main}

We remark that the assumptions on $X$ and $G$ in Theorem~\ref{TheoremMainTheoremOnCompleteIntersectionAndBireflection} and Theorem~\ref{TheoremYsep=n+k-1=>k-reflections} are satisfied both for a linear representation $X = V$ and for the case of multiplicative invariants where $X = \G_m^n$.

\vspace{0.2cm}

This article is organized as follows:

In Section \ref{SectionVarSepAndConnectedness} we study connectedness properties of the so-called separating variety $\Varsep$, leading to a characterization of the connectedness in a certain codimension of $\Varsep$ in relation to group elements which act as reflections (see Theorem~\ref{TheoremVsepConnectedInCodimIfAndOnlyIfGgenerated}).

Section \ref{SectionPolySepAlgebras} contains our main results concerning question (Q1). Using Hartshorne's connectedness theorem and the results of Section \ref{SectionVarSepAndConnectedness}, we can prove Theorem~\ref{TheoremYsep=nImpliesGeneratedByReflections}, which was mentioned above. Several examples show that the assumptions on $X$ and $G$ in Theorem~\ref{TheoremYsep=nImpliesGeneratedByReflections} cannot be dropped. Then we apply our results to the semigroup problem in multiplicative invariant theory and prove Theorem~\ref{TheoremMixedLaurentPolynomialRingInMIT}.

Section \ref{SectionCISepAlgebras} contains our main results concerning question (Q2). Here much deeper results from algebraic geometry about simply connected quotients and a purity theorem by Cutkosky are needed.

In Section \ref{SectionMinNumber} we prove Theorem~\ref{TheoremYsep=n+k-1=>k-reflections}. By using Theorem~\ref{TheoremVsepConnectedInCodimIfAndOnlyIfGgenerated} again, this time in combination with Grothendieck's connectedness theorem, we follow a similar path as in Section \ref{SectionPolySepAlgebras}.

\vspace{0.2cm}
\textbf{Acknowledgement.}
\thanks This paper contains the main results of my dissertation \cite{reimers2016dissertation}. I would like to express my sincere gratitude to my thesis advisor G.~Kemper, not only for introducing me to the subject, but also for his constant support and encouragement.

\section{The Separating Variety and Reflections}\label{SectionVarSepAndConnectedness}

Let $X$ be an $n$-dimensional $G$-variety as in the introduction.
The \emph{separating variety} of the action of $G$ on $X$ is defined to be the following subvariety of $X \times X$:
\[ \Varsep := \{ (x,y) \in X \times X \mid f(x) = f(y)  \text{ for all } f \in K[X]^G \}.\]
Since $G$ is finite, the invariants separate the orbits (see \cite[Section 2.3]{derksen2002computational}).
Thus two points $x,\, y \in X$ lie in the same orbit if and only if $(x,\,y) \in \Varsep$. So the separating variety really is the graph of the action of $G$ on $X$:
\[ \Varsep = \{ (x,\,\sigma x) \mid x \in X, \, \sigma \in G \}. \]
Hence 
$\Varsep$ is the union of all $H_{\sigma} := \{ (x,\, \sigma x) \mid x \in X\}$. Each $H_{\sigma}$ is an affine variety isomorphic to $X$, which leads to the following decomposition of $\Varsep$ into irreducible components.

\begin{prop}\label{PropositionIrreducibleComponentsOfVsep}
Assume that $X = \bigcup_{i=1}^r X_i$ is decomposed into its irreducible components $X_i$. Then for all $i$ and for all $\sigma \in G$ the subspace
\[ H_{\sigma,i} := \{ (x,\, \sigma x) \mid x \in X_i\} \subseteq X \times X\]
is an irreducible component of $\Varsep$, and $\Varsep$ is the union of all $H_{\sigma,i}$. 
\end{prop}

\begin{remark}\label{RemarkVsepEqudimensionalToo}
With the notation of Proposition~\ref{PropositionIrreducibleComponentsOfVsep} we also see:
\begin{enumerate}
\item[(a)]
It is $\dim(\Varsep) = n$.
\item[(b)] It is 
$\codim_{X \times X}(\Varsep)$ the minimum of all $\codim_{X \times X}(H_{\sigma,i})$. Since $$\codim_{X \times X}(H_{\sigma,i}) = \codim_{X_i \times \sigma X_i}(H_{\sigma,i}) = \dim X_i,$$ the formula
$$ \codim_{X \times X}(\Varsep) = \min \{ \dim(X_i) \mid i = 1 \upto r \}$$
holds. In particular, if $X$ is equidimensional, then $\codim_{X \times X}(\Varsep) = n$.
\end{enumerate}
\end{remark}

\begin{definition}\label{DefinitionKreflectionOnVariety}
Let $k$ be a non-negative integer. An element $\sigma \in G$ is called a \emph{$k$-reflection} (on $X$) if its fixed space $X^{\sigma} := \{ x \in X \mid \sigma x = x \}$ has codimension at most $k$ in $X$. For $k = 1$ we simply say that $\sigma$ is a reflection.
\end{definition}

\begin{remark}\label{RemarkKreflectionAndHeightOfCorrespondingIdeal} For $\sigma \in G$ let $I_{\sigma}$ be the ideal in $K[X]$ generated by all $\sigma f - f$ with $f \in K[X]$. Then $I_{\sigma}$ defines $X^{\sigma}$ as a subvariety of $X$, hence $\sigma$ is a $k$-reflection if and only if for the height of $I_{\sigma}$ we have $\height_{K[X]}(I_{\sigma}) \leq k$.
\end{remark}

Recall that a Noetherian topological space $Y$ is called \emph{connected in codimension $k$} (where $k \in \N_0$) if it satisfies one of the following two equivalent properties:
\begin{align}\label{FirstVersionOfConnectednessInCodim} \text{for all closed subsets $Z \subseteq Y$ with $\codim_Y(Z) > k$}\\ \notag \text{the space $Y \setminus Z$ is connected;} \end{align}
\begin{align}\label{SecondVersionOfConnectednessInCodim} \text{for all irreducible components $Y'$ and $Y''$ of $Y$} \\ \notag \text{there exists a finite sequence $Y_0 \upto Y_r$ of irreducible components of $Y$} \\ \notag \text{with $Y_0 = Y'$, $Y_r = Y''$ and $\codim_Y(Y_i \cap Y_{i+1}) \leq k \quad \text{for } i = 0 \upto r-1.$} \end{align}

Of course, being connected in codimension $k$ implies being connected in codimension $k+1$, so we have a chain of properties of $Y$. The strongest condition, $Y$ being connected in codimension $0$, is equivalent to $Y$ being irreducible. If $\dim(Y) < \infty$, then being connected in codimension $\dim(Y)$ simply means being connected.

It was shown by Dufresne \cite{dufresne2009separating} that for a linear action of $G$ on $X = K^n$ connectedness in codimension 1 of $\Varsep$ implies that the group is generated by 1-reflections. In the following theorem we extend this to non-linear actions and to $k > 1$, and we also add a converse, which will be needed later on.

\begin{thm}\label{TheoremVsepConnectedInCodimIfAndOnlyIfGgenerated}
The separating variety $\Varsep$ is connected in codimension $k$ if and only if $X$ is connected in codimension $k$ and $G$ is generated by $k$-reflections.
\end{thm}

\begin{proof}
Again, let $X = \bigcup_{i=1}^r X_i$ be decomposed into its irreducible components $X_i$, which leads to the components $H_{\sigma,i}$ of $\Varsep$ as seen in Proposition~\ref{PropositionIrreducibleComponentsOfVsep}. First, we look at the intersection of two components of $\Varsep$ to see which codimension arises. For $\sigma, \, \tau \in G$ and indices $i, \, j$ we have
\[ H_{\sigma, i} \cap H_{\tau, j} = \{ (x,y) \mid x \in X_i \cap X_j, \, y = \sigma x = \tau x \} \cong (X_i \cap X_j)^{\tau^{-1}\sigma}.\]
We know from Remark~\ref{RemarkVsepEqudimensionalToo} that $\dim(X) = n = \dim(\Varsep)$. In addition, we get
\begin{equation}\label{EquationCodimensionIntersection1} \codim_{\Varsep}(H_{\sigma, i} \cap H_{\tau, j}) = \codim_X((X_i \cap X_j)^{\tau^{-1}\sigma}).\end{equation}
Suppose $\Varsep$ is connected in codimension $k$. By assumption, for all $\sigma \in G$ and $i, j$ there exists a sequence of irreducible components $H_{\sigma_0, i_0},\, \dots,\, H_{\sigma_s, i_s}$ of $\Varsep$ with $i_0 = i,\, i_s = j,\, \sigma_0 = \iota$ (the neutral element of $G$), $\sigma_s = \sigma$ and 
\begin{equation}\label{EquationCodimensionIntersection2}\codim_{\Varsep}\left(H_{\sigma_l, i_l} \cap H_{\sigma_{l+1}, i_{l+1}}\right) \leq k  \quad \text{ for } 0 \leq l \leq s-1. \end{equation}
Putting (\ref{EquationCodimensionIntersection1}) and (\ref{EquationCodimensionIntersection2}) together leads to the inequality \begin{equation}\label{EquationCodimensionIntersection3}\codim_{X}\left(X_{i_l} \cap X_{i_{l+1}}\right)^{\sigma_{l}^{-1} \sigma_{l+1}}) \leq k  \quad \text{ for } 0 \leq l \leq s-1. \end{equation}
In particular, (\ref{EquationCodimensionIntersection3}) shows that $X_{i_l} \cap X_{i_{l+1}}$ has codimension $\leq k$. So we have a sequence of irreducible components from $X_{i_0} = X_i$ to $X_{i_s} = X_j$ that intersect in codimension $\leq k$, hence $X$ is connected in codimension $k$.

Moreover, (\ref{EquationCodimensionIntersection3}) implies that all $X^{\sigma_{l}^{-1} \sigma_{l+1}}$ have codimension $\leq k$, i.e., each $\sigma_{l}^{-1} \sigma_{l+1}$ is a $k$-reflection. Using $\sigma_0 = \iota$ and $\sigma_s = \sigma$ we can write 
\[\sigma = \sigma_0^{-1} \sigma_s = (\sigma_0^{-1} \sigma_1) \cdot (\sigma_1^{-1} \sigma_2) \cdot \dots \cdot (\sigma_{s-1}^{-1} \sigma_s)\]
as a product of $k$-reflections.

So we have proven the only-if-part by simply splitting (\ref{EquationCodimensionIntersection3}) into two weaker conclusions. It may therefore be surprising that the converse holds as well.

To prove it, let us start with indices $i,\, j$, and a sequence of components $X_{i_0},\, \dots,\, X_{i_s}$ with $X_i = X_{i_0}, \, X_j = X_{i_s}$ and $\codim_X(X_{i_l} \cap X_{i_{l+1}}) \leq k$. Consequently, for $\sigma \in G$ we know from (\ref{EquationCodimensionIntersection1}), that all $H_{\sigma, i_l} \cap H_{\sigma, i_{l+1}}$ have codimension $\leq k$. So we already have a sequence from $H_{\sigma, i}$ to $H_{\sigma, j}$ as desired.

Now take two elements $\sigma', \, \sigma'' \in G$. By assumption, there exist $k$-reflections $\tau_1,\, \ldots,\, \tau_s \in G$ with $(\sigma')^{-1} \sigma'' = \tau_1 \cdot \ldots \cdot \tau_s$.
Since for all $l$ we have
\[\min \{ \codim_X(X_m^{\tau_l}) \mid m = 1 \upto r \} = \codim_X(X^{\tau_l}) \leq k,\]
for each $\tau_l$ there exists an $i_l$ such that
\begin{equation}\label{EquationCodimensionIntersection4} \codim_X(X_{i_l}^{\tau_l}) \leq k. \end{equation}
If we write $\sigma_0 := \sigma'$ and $\sigma_l := \sigma_{l-1} \tau_l$ for $l = 1 \upto s$, then
\[ \sigma_s = \sigma_0 \cdot \tau_1 \cdot \ldots \cdot \tau_s = \sigma' \cdot (\sigma'^{-1} \cdot \sigma'') = \sigma''.\]
It follows from (\ref{EquationCodimensionIntersection1}) together with (\ref{EquationCodimensionIntersection4}) that
\[\codim_{\Varsep}(H_{\sigma_{l-1}, i_l} \cap H_{\sigma_{l}, i_l}) = \codim_{X}((X_{i_l})^{\sigma_{l-1}^{-1} \sigma_l}) =  \codim_X((X_{i_l})^{\tau_l}) \leq k.\]
We already saw how to construct a sequence of components from every $H_{\sigma, i_l}$ to $H_{\sigma, i_{l+1}}$ as desired. Putting these together, for all $i,\, j$ we can construct a sequence
\[H_{\sigma_0, i},\, \ldots ,\, H_{\sigma_0, i_1}, H_{\sigma_1, i_1},\, \ldots ,\, H_{\sigma_1, i_2}, H_{\sigma_2, i_2},\, \ldots ,\, H_{\sigma_s, i_s},\, \ldots ,\, H_{\sigma_s, j},\]
from $H_{\sigma', i}$ to $H_{\sigma'', j}$ such that two successive components intersect in codimension $\leq k$.
\end{proof}

We specialize Theorem~\ref{TheoremVsepConnectedInCodimIfAndOnlyIfGgenerated} to the case $k = n = \dim(X)$.

\begin{cor} \label{CorollaryVsepConnectedIfAndOnlyIf}
The separating variety $\Varsep$ is connected if and only if $X$ is connected and $G$ is generated by elements having a fixed point.
\end{cor}

Let us finish this section with a few remarks about ideals that define $\Varsep$. When studying the separating property, the following map naturally comes into play:
\begin{equation}\label{EquationDefinitionOfDeltaMap} \delta: K[X] \to K[X] \otimes_K K[X], \quad f \mapsto f \otimes 1 - 1 \otimes f.\end{equation}
Following the notation of \cite{dufresne2009separating} we will refer to (\ref{EquationDefinitionOfDeltaMap}) as the $\delta$-map throughout this paper. It is obviously $K$-linear and a quick calculation shows that
\[ \delta(fg) = \delta(f) (g \otimes 1) + (f \otimes 1) \delta(g) - \delta(f) \delta(g)\]
for all $f,\,g \in K[X]$. So in particular, $\delta(fg)$ lies in the ideal generated by $\delta(f)$ and $\delta(g)$ in $K[X] \otimes_K K[X]$. This shows that the images under $\delta$ of all generating sets of invariants generate the same ideal, which we will denote by $\Isep$, i.e., \[\Isep  := (\delta(K[X]^G))_{K[X] \otimes_K K[X]}.\]
The definition of $\Varsep$ now reads as
\[ \Varsep := \Var_{X \times X}(\Isep) := \{ (x,\,y) \in X \times X \mid h(x,\,y) = 0 \text{ for all } h \in \Isep \}.\]
Furthermore, a set of invariants $S \subseteq K[X]^G$ is now separating if and only if the image of $S$ under the $\delta$-map defines the subvariety $\Varsep \subseteq X \times X$ as its vanishing set. By Hilbert's Nullstellensatz, this is equivalent to the condition that the radical of the ideal generated by $\delta(S)$ and the radical of $\Isep$ conincide. Thus we have proved:

\begin{prop}\label{PropositionSeparatingRadicalCondition} A subset $S \subseteq K[X]^G$ is separating if and only if $\sqrt{(\delta(S))} = \sqrt{\Isep}$.
\end{prop}

\section{Polynomial Separating Algebras}\label{SectionPolySepAlgebras}

Hartshorne's connectedness theorem \cite{hartshorne1962complete} is a main ingredient to the results of this section.
It relates the connectedness property of the spectrum of a Noetherian ring $R$ to the depths of localizations of $R$. We will use it in a version with the \emph{Cohen-Macaulay defect}. Recall that the Cohen-Macaulay defect of a Noetherian local ring $R$ is defined to be
\[ \cmdef(R) := \dim(R) - \depth(R) \in \N_0.\] 
More generally, for a Noetherian ring $R$ the Cohen-Macaulay defect is defined to be
\[\cmdef(R) := \sup\, \{\, \cmdef(R_{\p}) \mid \p \in \Spec(R) \,\} \in \N_0 \cup \{\infty\}, \]
which can be shown to be consistent with the definition of $\cmdef(R)$ for a local ring.

\begin{thm} \emph{\textbf{(Hartshorne \cite{hartshorne1962complete})}}\label{TheoremHartshornesConnectednessWithCMDEF}
Let $R$ be a Noetherian ring. Assume that $\Spec(R)$ is connected and that $k := \cmdef(R)$ is finite. Then $\Spec(R)$ is connected in codimension $k + 1$. 
\end{thm}

\begin{proof}
We have the following condition on the prime ideals $\p$ of $R$: If $\height_R(\p) > k + 1$, then
\[ k \geq \height_R(\p) - \depth(R_{\p}) \geq k + 2 - \depth(R_{\p}),\]
and hence $\depth(R_{\p}) \geq 2$. In \cite[Corollary 2.3]{hartshorne1962complete} it is shown that $\Spec(R)$ is \emph{locally} connected in codimension $k$ under this condition. If we require $\Spec(R)$ to be connected, this implies that $\Spec(R)$ is connected in codimension $k$ (see \cite[Remark 1.3.2]{hartshorne1962complete}).
\end{proof}

Now again let $X$ be a $G$-variety. Combining Hartshorne's connectedness theorem with Theorem~\ref{TheoremVsepConnectedInCodimIfAndOnlyIfGgenerated} leads to the following.

\begin{thm}\label{TheoremCMDEFimpliesReflectionGroup} Assume that $X$ is connected and that $G$ is generated by elements having a fixed point. Write $R = K[X] \otimes_K K[X]$ and define
\[ k := \min\, \{ \cmdef( R/J) \mid J \subseteq R \text{ an ideal with } \sqrt{J} = \sqrt{\Isep} \} .\] 
Then $G$ is generated by $(k+1)$-reflections.
\end{thm}

\begin{proof}
The assumptions on $X$ and the action of $G$ imply that $\Varsep$ is connected by Corollary~\ref{CorollaryVsepConnectedIfAndOnlyIf}. Let $J$ be an ideal in $R$ with $\sqrt{J} = \sqrt{\Isep}$ and $k = \cmdef(R/J)$.
Theorem~\ref{TheoremHartshornesConnectednessWithCMDEF} tells us now that $\Spec(R/J)$, which is homeomorphic to $\Spec(K[\Varsep])$, is connected in codimension $k+1$. Of course, it is equivalent to say that $\Varsep$ is connected in codimension $k+1$. Therefore, by Theorem~\ref{TheoremVsepConnectedInCodimIfAndOnlyIfGgenerated}, 
$G$ is generated by $(k+1)$-reflections.
\end{proof}

Example~\ref{ExampleDifferentCMDEF} will show that $\Isep$ need not be radical and that neither $\Isep$ nor $\sqrt{\Isep}$ must have the smallest 
Cohen-Macaulay defect among all ideals $J \subseteq K[X] \otimes_K K[X]$ with $\sqrt{J} = \sqrt{\Isep}$. In particular, the number $k$ in Theorem~\ref{TheoremCMDEFimpliesReflectionGroup} need not be the Cohen-Macaulay defect of $K[\Varsep]$. Since an ideal is called set-theoretically Cohen-Macaulay if there exists a Cohen-Macaulay ideal with the same radical (cf. \cite{singh2007local}), we propose to call this number the \emph{set-theoretical Cohen-Macaulay defect} of $\Isep$ (or $\Varsep$). 

To the best of the author's knowledge, no algorithm is known to compute the set-theoretical Cohen-Macaulay defect of $\Isep$. In several examples (like Example~\ref{ExampleDifferentCMDEF}) it indeed coincides with the minimal number $l$ such that $G$ is generated by $(l+1)$-reflections. Based on these examples we make the following conjecture.

\begin{conj} Assume the notation and hypotheses of Theorem~\ref{TheoremCMDEFimpliesReflectionGroup}. Then \[ k  = \min \{ l \in \N_0 \mid G \text{ is generated by $(l+1)$-reflections}\}.\]
\end{conj}

Now we can prove the main result about separating algebras with $n$ algebraically independent generators.

\begin{thm}\label{TheoremYsep=nImpliesGeneratedByReflections} Assume that $X$ is connected and Cohen-Macaulay and that $G$ is generated by elements having a fixed point. If $\gammasep = n$, then $G$ is generated by reflections.
\end{thm}

\begin{proof}
Since $X$ is Cohen-Macaulay, it follows that $X \times X$ is Cohen-Macaulay, too (see \cite{watanabe1969tensor}). In addition, $X$ is connected, so $X$ and $X \times X$ are also equidimensional, since local Cohen-Macaulay rings are equidimensional (see \cite[Corollary 18.11]{eisenbud1995commutative}).

Now let $\{ f_1,\, \ldots ,\, f_n \}$ be a set of separating invariants. Using the $\delta$-map (defined in (\ref{EquationDefinitionOfDeltaMap})) this set defines the following ideal in $K[X] \otimes_K K[X]$:
\begin{equation}\label{EquationDefinitionOfJ} J := \left( \delta(f_1) \upto \delta(f_n) \right), \end{equation}
which has the same radical as $\Isep$ by Proposition~\ref{PropositionSeparatingRadicalCondition}. Hence we have
\begin{equation}\label{EquationHeightOfJ} \height(J) = \height(\sqrt{\Isep}) = \codim_{X \times X}(\Varsep) =  n,\end{equation}
by Remark~\ref{RemarkVsepEqudimensionalToo}.
Since $K[X] \otimes_K K[X]$ is Cohen-Macaulay, (\ref{EquationDefinitionOfJ}) and (\ref{EquationHeightOfJ}) imply that $J$ is generated by a $(K[X] \otimes_K K[X])$-regular sequence. Therefore, $(K[X] \otimes_K K[X])/J$ is Cohen-Macaulay as well (see \cite[Proposition 18.13]{eisenbud1995commutative}). Now we can use Theorem~\ref{TheoremCMDEFimpliesReflectionGroup} with $k = 0$. 
\end{proof}

Dufresne \cite{dufresne2009separating} gave an example of a representation for which the invariant ring is not a polynomial ring, but still $\gammasep$ equals $n$. This suggested that the choice of $J$ in Theorem~\ref{TheoremCMDEFimpliesReflectionGroup} matters. The following example illustrates this point as it results in various Cohen-Macaulay defects. 
It is taken from the database of invariant rings of Kemper et al. \cite{kemper2001database}.

\begin{example}\cite[ID 10253]{kemper2001database}\label{ExampleDifferentCMDEF}
Let $\characteristic(K) = 2$. We look at the following subgroup, isomorphic to $(\Z/2\Z)^3$, of $\GL_4(K)$: 
\[ G := \{ \begin{pmatrix}1&0&0&0\\b&1&c&0\\0&0&1&0\\c&0&a&1\end{pmatrix} \mid a,\,b,\,c \in \F_2 \} \subseteq \GL_4(K).\]
Its natural action on $V = K^4$ is generated by reflections.
Using the computer algebra system \textsc{magma} \cite{bosma1997magma}, we have computed the primary invariants
\begin{eqnarray*}
f_1 &:=& x_1,\\
f_2 &:=& x_3,\\
f_3 &:=& x_1^2 x_3 x_4 + x_1^2 x_4^2 + x_1 x_3^2 x_4 + x_1 x_3 x_4^2 + x_3^2 x_4^2 + x_4^4,\\
f_4 &:=& x_1^3 x_2 + x_1 x_2 x_3^2 + x_1 x_3^2 x_4 + x_1 x_3 x_4^2 + x_2^4 + x_2^2 x_3^2 + x_3^3 x_4 + x_3^2 x_4^2,
\end{eqnarray*}
and a secondary invariant
\[h  := x_1^2 x_2 + x_1 x_2^2 + x_3^2 x_4 + x_3 x_4^2.\]
Hence, the invariant ring 
\[ K[V]^G = K[x_1,\,x_2,\,x_3,\,x_4]^G = K[f_1,\,f_2,\,f_3,\,f_4,\,h] \]
is not a polynomial ring. Between the generating invariants there is the relation
\[ f_1^3 h + f_1^2 f_3 + f_1 f_2^2 h + f_2^2 f_4 + h^2 = 0.\]
So by defining $g_3 := f_1 h + f_3$ and $g_4 := f_1 h + f_4$, we get $$h^2 = f_1^2 g_3 + f_2^2 g_4.$$ Since $\characteristic(K) = 2$, we see from this relation that the values of $f_1,\, f_2, \,g_3,\, g_4$ at a point $x \in K^4$ determine $h(x)$. From the definition of $g_3$ and $g_4$ it is clear that the values of $f_3$ and $f_4$ at $x$ are also determined by this. Hence $S := \{ f_1,\,f_2,\,g_3,\,g_4 \}$ is separating.
In this example $\Isep$ is not a radical ideal. Let $J$ be the ideal in $R := K[V] \otimes_K K[V]$ generated by $\delta(S)$. Using the graded version of the Auslander-Buchsbaum formula, we calculated the following Cohen-Macaulay defects with \textsc{magma}:
\[ \cmdef(R/\Isep) = 2, \quad \cmdef(R/\sqrt{\Isep}) = 1, \quad \cmdef(R/J) = 0.\]
Of course, $\cmdef(R/J) = 0$ is not surprising, as it was used in Theorem~\ref{TheoremYsep=nImpliesGeneratedByReflections}.\hfill$\triangleleft$
\end{example}

Let us look at the assumptions in Theorem~\ref{TheoremYsep=nImpliesGeneratedByReflections} more closely. Of course, any example of a (non-trivial) free group action of $G$ on $X$ with an invariant ring isomorphic to a polynomial ring shows that the assumption that $G$ has fixed points cannot be dropped from Theorem~\ref{TheoremYsep=nImpliesGeneratedByReflections}.

\begin{example}
Let $\characteristic(K) = p > 0$, and let $G = \Z / p \Z$ be the cyclic group of order $p$. When we look at the additive action of $\Z / p \Z$ on $V = K$ via $(\sigma,\, x) \mapsto \sigma + x$, we see that 
\[K[V]^G = K[x]^G = K[x^p - x] \]
is a polynomial ring. But a non-zero group element $\sigma \in \F_p$ does not have a fixed point, so in particular, $G$ is not a reflection group.\hfill$\triangleleft$
\end{example}

The next example shows that the assumption that $X$ is Cohen-Macaulay cannot be dropped from Theorem~\ref{TheoremYsep=nImpliesGeneratedByReflections}.

\begin{example}\label{ExampleTwoAffinePlanesIntersectingInAPoint}
Let $\characteristic(K) \neq 2$ and consider the affine variety
\[X := \Var (x_1^2 - x_3^2, \, x_2^2 - x_4^2,\, x_1x_2 - x_3x_4, \, x_1x_4 - x_2x_3) \subseteq K^4, \]
which is the union of two planes intersecting at the origin:
\[X = \Var (x_1 - x_3, \, x_2 - x_4 ) \cup \Var (x_1 + x_3, \, x_2 + x_4 ).\]
Hartshorne's connectedness theorem in the form of Theorem~\ref{TheoremHartshornesConnectednessWithCMDEF} now tells us that $X$ is not Cohen-Macaulay at the intersection point, since it is not connected in codimension 1 there.

A cyclic group $G = \langle \sigma \rangle$ of order 2 acts (on $V = K^4$ and) on $X$ by
\begin{equation*} \sigma \cdot (x_1,\,x_2,\,x_3,\,x_4) := (x_1,\,x_2,\,-x_3,\,-x_4).\end{equation*}
The action of $G$ on $X$ interchanges the two planes while fixing the origin. So the generator of $G$ is a 2-reflection on $X$.

The invariant ring of the representation $V = K^4$ of $G$ can be easily seen to be
\[ K[V]^G = K[x_1,\,x_2,\,x_3^2,\, x_3x_4, \, x_4^2].\]
Since we are in a non-modular case, the finite group $G$ is linearly reductive. Therefore, $K[X]^G$ is the quotient ring of $K[V]^G$ modulo the vanishing ideal of $X$.
We get 
\[ K[X]^G = K[\overline{x}_1,\,\overline{x}_2].\] 
So the invariant ring of the action on $X$ is a polynomial ring, but in contrast to Theorem~\ref{TheoremYsep=nImpliesGeneratedByReflections}, $G$ is not generated by reflections.\hfill$\triangleleft$
\end{example}

As outlined in the introduction, we finish this section on polynomial separating algebras with an application to multiplicative invariant theory. So now let $L$ be a lattice of rank $n$ with an action of $G$ on $L$ by automorphisms, and let $K[L]$ be the group ring of $L$ over $K$ which carries an induced action of $G$ by $K$-algebra automorphisms. Since \[ K[L] \cong K[x_1^{\pm 1}, \,\ldots, \, x_n^{\pm 1}]\] is a Laurent polynomial ring over $K$ in $n$ indeterminates, the corresponding $G$-variety is
\begin{equation}\label{multInvTheoDefOfVariety} X := \left(\G_m\right)^n = \mathcal{V}(x_1 y_1 - 1, \,\ldots, \, x_n  y_n -1 ) \subseteq K^{2n},\end{equation}
i.e., $X$ is an $n$-dimensional algebraic torus.

In multiplicative invariant theory an element $\sigma \in G$ is called a $k$-reflection if the sublattice $\{ \sigma l - l \mid l \in L\}$ has rank at most $k$ (see \cite[Section 1.7]{lorenz2005multiplicative}). But by \cite[Lemma 4.5.1]{lorenz2005multiplicative}, this is equivalent to the condition that the ideal $(\sigma f - f \mid f \in K[L])$ in $K[L] = K[X]$ has height at most $k$, so that a reflection on $L$ is exactly a reflection on $X$ (see Remark~\ref{RemarkKreflectionAndHeightOfCorrespondingIdeal}).

Now we make the following contribution to the semigroup problem in multiplicative invariant theory.

\begin{thm}\label{TheoremMixedLaurentPolynomialRingInMIT}
Let $L$ be a lattice and let $G$ be a finite group acting on $L$ by automorphisms. If $K[L]^G$ is isomorphic to a mixed Laurent polynomial ring, then $G$ is generated by reflections.
\end{thm}

\begin{proof}
Since $\G_m$ is a connected linear algebraic group, the affine variety $X$ in (\ref{multInvTheoDefOfVariety}) is irreducible and non-singular (hence Cohen-Macaulay). Furthermore, every $\sigma \in G$ fixes the point $(1 \upto 1) \in X$. Hence the prerequisites of Theorem~\ref{TheoremYsep=nImpliesGeneratedByReflections} are satisfied. By assumption, we have 
\[ K[X]^G = K[L]^G = K[f_1^{\pm 1}, \, \ldots, \, f_k^{\pm 1}, \, f_{k+1}, \, \ldots, f_n] \]
with invariants $f_i \in K[X]$ (where obviously $n = \trdeg_K(K[X]^G) = \dim(X))$. But from this generating set of invariants we easily extract the smaller separating set
\[ S = \{ f_1, \ldots, f_n \},\]
since the inverses of $f_1 \upto f_k$ are not needed to separate the orbits. So we have $\gammasep = n$ and the result follows with Theorem~\ref{TheoremYsep=nImpliesGeneratedByReflections}. 
\end{proof}

\section{Complete Intersection Separating Algebras}\label{SectionCISepAlgebras}

Vinberg's lemma\index{Vinberg's lemma} \cite[Lemma 2]{kac1982finite} is one ingredient to the proof of Kac-Watanabe's theorem about complete intersection invariant rings. Roughly speaking, it states that if a finite group acts on a sufficiently nice topological space such that the quotient is simply connected, then the group must be generated by elements having a fixed point. A version for the Euclidean topology of complex algebraic varieties appears in \cite[Section 8.3]{popov1994invariant}. A version for schemes which is designed for a generalization of Dufresne's and Kac-Watanabe's results will be given below, in Lemma~\ref{LemmaVinberg}, after some preliminary remarks.

When we look at the action of a group $G$ on a scheme $\X$ by morphisms, an element $x \in \X$ should be considered a fixed point of a group element $\sigma \in G$ if and only if $\sigma x = x$ and $\sigma$ acts as identity on the residue field $\kappa(x)$ of $x$. If $\X$ is a \emph{separated} scheme, then the set of fixed points $\X^{\sigma}$ (in the above sense) of $\sigma \in G$ is always a closed subscheme of $\X$ (see \cite[Chapter 9]{gortz2010algebraic}). For example if $\X = \Spec(R)$ is an affine scheme, then 
\[ \Spec(R)^{\sigma} = \{ \pprime \in \Spec(R) \mid \sigma f - f \in \pprime \text{ for all } f \in R \}.\]
So Definition \ref{DefinitionKreflectionOnVariety} about $k$-reflections carries over to actions on separated schemes, and, for the affine $G$-variety $X$, Remark~\ref{RemarkKreflectionAndHeightOfCorrespondingIdeal} shows that an element $\sigma \in G$ is a $k$-reflection on $X$ if and only if it is a $k$-reflection on the scheme $\X = \Spec(K[X])$.

We also need some general facts about the quotient of a scheme $\X$ by a finite group $G$. Following \cite{fu2011etale} we will call the action \emph{admissible} if there exists an affine $G$-invariant morphism $\pi : \X \to Y$ such that $\mathcal{O}_{Y} \cong (\pi_{\ast} \mathcal{O}_{\X})^G$. Then $Y$ is not only the categorical but also the geometric quotient of $\X$ by $G$ (see \cite[p. 119]{fu2011etale}). Moreover, every open subset $V \subseteq Y$ is the quotient of $\pi^{-1} (V) \subseteq \X$.

Finally, we recall what it means for a 
scheme $Z$ to be simply connected. An \emph{\'etale covering} of $Z$ is a scheme $Z'$ together with a finite \'etale morphism $f: Z' \to Z$. It is called trivial if $Z'$ is a finite disjoint union of open subschemes which are all isomorphic to $Z$ via $f$. And $Z$ is called \emph{simply connected} if every \'etale covering of $Z$ is trivial (see \cite[Chapter 4, Section 2.2]{danilov1996cohomology}).

\begin{lem}\label{LemmaVinberg}
Let $\X$ be a separated scheme, connected in codimension $k$, on which a finite group $G$ acts admissibly. Let $\pi : \X \to Y$ be the quotient and suppose that $\X$ is of finite presentation over $Y$. Furthermore, assume that $Y$ has the following property:
\begin{align}\label{PropertySimplyConnectedInCodimensionk}
\text{for all closed subsets $Z$ of $Y$ with $\codim_Y(Z) > k$}\\ \notag \text{the space $Y \setminus Z$ is simply connected.}
\end{align}
Then $G$ is generated by $k$-reflections on $\X$.
\end{lem}

Due to the similarity to (\ref{FirstVersionOfConnectednessInCodim}), we will refer to (\ref{PropertySimplyConnectedInCodimensionk}) as \emph{simply connected in codimension $k$}, although this is not a standard term. In the case $k = 1$, Popov and Vinberg call this property strongly simply connected (see \cite[Proposition 8.3]{popov1994invariant}).

\begin{proof}
Since $\X$ is separated, the finite union  
\[ L := \bigcup\limits_{\substack{\sigma \in G\\ \codim_{\X}(\X^{\sigma}) > k}} \X^{\sigma} \]
is a closed subset of $\X$. In addition, $L$ is $G$-stable (since for all $\tau, \, \sigma \in G$ we have $\tau \X^{\sigma} = \X^{\tau \sigma \tau^{-1}}$), and has $\codim_{\X}(L) > k$. So by assumption, $\widetilde{\X} := \X \setminus L$ is connected. 

As $\pi$ is integral (see \cite[Proposition 3.1.1]{fu2011etale}), $\pi(L)$ is closed and $\codim_Y(\pi(L)) > k$ (cf. \cite[Proposition 12.12]{gortz2010algebraic}). So by assumption (\ref{PropertySimplyConnectedInCodimensionk}), $\widetilde{Y} := Y \setminus \pi(L)$ is simply connected. 

Now $G$ acts on $\widetilde{\X} := \X \setminus L$ with quotient morphism $\widetilde{\pi}: \widetilde{\X} \to \widetilde{Y}$. An element $\sigma \in G$ is a $k$-reflection on $\X$ if and only if it has a fixed point in $\widetilde{\X}$. So we have to show that $G$ is equal to the subgroup $H := \langle \sigma \in G \mid \widetilde{\X}^{\sigma} \neq \emptyset \rangle$.

Since $H$ contains the inertia subgroups $I_{\widetilde{x}} := \{ \sigma \in G_{\widetilde{x}} \mid \sigma = \id \text{ on } \kappa(\widetilde{x})\}$ of all points $\widetilde{x} \in \widetilde{\X}$, the induced morphism
\[ \varphi: \widetilde{\X} / H \to \widetilde{Y} \]
is \'etale by \cite[Satz 4.2.1]{kurke1975henselsche} or \cite[Lemma 4.11]{reimers2016dissertation}. As $\widetilde{Y}$ is simply connected, $\widetilde{\X}/H$ is therefore isomorphic to a disjoint union of finitely many copies of $\widetilde{Y}$. But $\widetilde{\X}$ and therefore $\widetilde{\X}/H$ are connected, so there is only one copy, and $\varphi$ is an isomorphism. This shows that the quotients of $\widetilde{\X}$ by $G$ and by $H$ are the same. 

Now let $\sigma \in G$. To show that $\sigma$ lies in $H$, take any point $\widetilde{x} \in \widetilde{\X}$. Since the $G$-orbit of this point is the same as the $H$-orbit, there exists  $\tau \in H$ with  $\sigma \tau \in G_{\widetilde{x}}$. For the quotient $\widetilde{\pi}$ of a scheme by a finite group, it is a general fact that the canonical homomorphism from the stabilizer of $\widetilde{x}$ to the automorphism group of the field extension $\kappa(\widetilde{\pi}(\widetilde{x})) \subseteq \kappa(\widetilde{x})$ is surjective (see part (iii) of \cite[Proposition 3.1.1]{fu2011etale}). And this holds now for both the stabilizer in $G$ and in $H$.
Thus for $\sigma \tau \in G_{\widetilde{x}}$ there exists $\mu \in H_{\widetilde{x}}$ such that $\sigma \tau \mu$ is mapped to the identity element of the Galois group (i.e., lies in the inertia subgroup of $\widetilde{x}$). So $\sigma \tau \mu \in I_{\widetilde{x}} \subseteq H$, and hence $\sigma \in H$ follows.
\end{proof}

The next step is to derive property (3.1) from purity theorems. First recall that every scheme morphism $h: Y_1  \to Y_2$ induces a functor $h^{\ast}$ from the category of \'etale coverings of $Y_2$ to the category of \'etale coverings of $Y_1$ by pulling back: 
\begin{equation}\label{EquationDefiningFunctorPullBack} g: V \xrightarrow{\text{et}} Y_2 \quad \Longrightarrow \quad h^{\ast}(g):  V \times_{Y_2} Y_1 \xrightarrow{\text{et}} Y_1.\end{equation}

\begin{remark}\label{RemarkEquivalenceOfCategoriesAndSimplyConnected}
Suppose that $h^{\ast}$ is an equivalence of categories, and that $Y_2$ is simply connected. Since disjoint union and fiber product commute, it follows then that $Y_1$ is simply connected, too.
\end{remark}

A pair of a scheme $Y$ and a closed subscheme $Z$ is now called \emph{pure} (see \cite[X, D\'efinition 3.1]{grothendieck1962cohomologie}) if for all open subschemes $U$ of $Y$ the functor
$i^{\ast}$ induced by the inclusion $i: (Y \setminus Z) \cap U \hookrightarrow U$ is an equivalence of categories. Moreover, a Noetherian local ring $(R,\, \mmax)$ is called pure if the pair $(\Spec(R),\, \{\mmax\})$ is pure.

\begin{remark}\label{RemarkPurityImpliesSimplyConnected}
Suppose that $(Y,\,Z)$ is pure, and that $Y$ is simply connected. It follows then by Remark~\ref{RemarkEquivalenceOfCategoriesAndSimplyConnected} that $Y \setminus Z$ is simply connected, too.
\end{remark}

Grothendieck \cite{grothendieck1962cohomologie} proved that local complete intersection rings of dimension $\geq 3$ are pure. This was extended by Cutkosky \cite{cutkosky1995purity} to a larger class of rings. The property ''complete intersection'' is generalized in two ways for this. 

First there is the notion of the \emph{complete intersection defect} $\cid(R)$ of a Noetherian local ring $R$. This can be intrinsically defined via the Koszul complex of $R$. But for us it is enough to restrict ourselves to situations where $R = S/I$ is the quotient of a regular local ring $S$ by an ideal $I$: Then $\cid(R)$ equals the minimal number of generators of $I$ minus $\height_S(I)$ (see \cite[Satz 1]{kiehl1965vollstandige}).

Secondly, recall that a Noetherian ring is called a \emph{complete intersection in codimension $k$} if for all prime ideals $\pprime \in \Spec(R)$ with $\height_R(\pprime) \leq k$ the localization $R_{\pprime}$ is a complete intersection.

Both these generalizations of the complete intersection property appear in combination in the following theorem. The purity theorem of Cutkosky \cite{cutkosky1995purity} is the main ingredient for the proof.

\begin{thm}\label{TheoremCIisConnectedInCodim2}
Let $R$ be a Noetherian local ring. Assume that $Y = \Spec(R)$ is simply connected, and that $R$ is excellent, a quotient of a regular local ring, equidimensional, and a complete intersection in codimension $2 + \cidef(R)$.
Then $Y$ is simply connected in codimension $2$.
\end{thm}

\begin{proof}
Let $Z \subseteq Y$ be a closed subscheme of codimension larger than $2$. If we can show that the pair $(Y,\,Z)$ is pure, then the result follows with Remark~\ref{RemarkPurityImpliesSimplyConnected}. For that, by \cite[X, Proposition 3.3]{grothendieck1962cohomologie}, it is to show that all local rings $R_{\pprime}$ with $\pprime \in Z$ are pure. But for such a prime ideal $\pprime$ we have
\[ \dim(R_{\pprime}) = \height_R(\pprime) = \codim_Y(\overline{\{\pprime\}}) \geq \codim_Y(Z) \geq 3.\]
The assumptions on $R$ imply that $R_{\pprime}$ is also excellent, a quotient of a regular local ring, equidimensional, and a complete intersection in codimension $2 + \cidef(R_{\pprime})$. Thus it follows by the purity theorem of Cutkosky \cite[Theorem 19]{cutkosky1995purity} that $R_{\pprime}$ is pure.
\end{proof}

For finitely generated $K$-algebras there is the following (global) complete intersection defect.

\begin{definition}\label{DefinitionCIDEFECT}
Let $A$ be an affine $K$-algebra of dimension $n$. The \emph{(global) complete intersection defect} $\cid(A)$ of $A$ \emph{(over $K$)} is the smallest number $l \in \N_0$ such that there exists an $m \in \N_0$ and a presentation $$A \cong K[x_1 \upto x_m] / (f_1 \upto f_{m-n+l})$$ with polynomials $f_i \in K[x_1 \upto x_m]$.
\end{definition}

It is easy to see that for all prime ideals $\pprime$ of the affine algebra $A$ the (local) complete intersection defect $\cid(A_{\pprime})$ is less or equal to $\cid(A)$.

\begin{prop}\label{PropositionAlgebraWithCIImpliesSpecOfComplIsSimplConnecInCodim2}
Let $A$ be an affine $K$-domain that is a complete intersection in codimension $2 + \cid(A)$, where $\cid(A)$ is the global complete intersection defect of $A$. Let $\mmax$ be a maximal ideal of $A$ and $\widehat{A}$ the $\mmax$-adic completion of $A$. Then $\Spec(\widehat{A})$ is simply connected in codimension 2.
\end{prop}

\begin{proof}
We need to show that $\widehat{A}$ satisfies the assumption for $R$ in Theorem~\ref{TheoremCIisConnectedInCodim2}.

It is clear that $\widehat{A}$ is excellent and a quotient of a regular local ring. While it need not be an integral domain, it is however equidimensional by \cite[Corollary after Theorem 31.5]{matsumura1989commutative}. Moreover, Lemma~\ref{LemmaCompletionIsCIinCodimToo} below shows that $\widehat{A}$ is a complete intersection in codimension $2 + \cid(A)$. The completion $\widehat{A}$ can be viewed as the completion of the localization $A_{\mmax}$, and completion of local rings preserves the (local) complete intersection defect (see \cite[Lemma 1]{kiehl1965vollstandige}), hence
\[\cid(\widehat{A}) = \cid(A_{\mmax}) \leq \cid(A).\]

So a fortiori $\widehat{A}$ is a complete intersection in codimension $2 + \cid(\widehat{A})$.
Furthermore, $\Spec(\widehat{A})$ is simply connected as it is the spectrum of a strictly Henselian local ring (see \cite[I, Example 5.2(c)]{milne1980etale}). So the claim now follows with Theorem~\ref{TheoremCIisConnectedInCodim2}.
\end{proof}

\begin{lem}\label{LemmaCompletionIsCIinCodimToo}
Let $S$ be a regular ring and let $I \subseteq S$ be a prime ideal, such that $R := S/I$ is a complete intersection in codimension $k$. Moreover, let $\mathfrak{n}$ be a maximal ideal of $S$ with $I \subseteq \mathfrak{n}$, and let $\mmax := \nmax /I$. Then the $\mmax$-adic completion $\widehat{R}$ of $R$ is a complete intersection in codimension $k$, too.
\end{lem}

\begin{proof}
We can view $\widehat{R}$ as $\widehat{S} / I \widehat{S}$ where $\widehat{S}$ is the $\nmax$-adic completion of $S$. Now let $\qprime \in \Spec(\widehat{S})$ with $I \widehat{S} \subseteq \qprime$ and $\height_{\widehat{S} / I \widehat{S}}(\qprime / I \widehat{S}) \leq k$. We need to show that $(\widehat{S} / I \widehat{S})_{\qprime / I \widehat{S}}$, which is isomorphic to $\widehat{S}_{\qprime} / (I \widehat{S})_{\qprime}$, is a complete intersection ring.

Since $\widehat{S}_{\qprime}$ is a regular local ring, it is precisely to show that $(I \widehat{S})_{\qprime}$ is a complete intersection ideal (i.e., generated by $\height_{\widehat{S}_{\qprime}}((I\widehat{S})_{\qprime})$ many elements). Let $\varepsilon_R$ and $\varepsilon_S$ denote the canonical ring maps $R \to \widehat{R}$ and $S \to \widehat{S}$, respectively. Let $\pprime := \varepsilon_S^{-1}(\qprime) \in \Spec(S)$. Then $I \subseteq \pprime$ and $\varepsilon_R^{-1}(\qprime / I \widehat{S}) = \pprime / I$. Since $\varepsilon_R$ satisfies going-down (see \cite[Theorem 8.8. \& Theorem 9.5]{matsumura1989commutative}), we have
\[\height_R(\pprime/ I) \leq \height_{\widehat{S} / I \widehat{S}}(\qprime / I \widehat{S}) \leq k.\]
So by assumption on $R$, it is $R_{\pprime/I} \cong S_{\pprime}/I_{\pprime}$ a complete intersection ring, hence $I_{\pprime}$ is a complete intersection ideal. Therefore, there exist $a_1 \upto a_l \in I$ with $I_{\pprime} = (a_1 \upto a_l)_{S_{\pprime}}$ and $l = \height_{S_{\pprime}}(I_{\pprime})$. This means that for every $a \in I$ there exists an $s \in S \setminus \pprime$ such that $s a \in (a_1 \upto a_l)_S$. But this also shows that $(I \widehat{S})_{\qprime}$ is generated by $\varepsilon_S(a_1) \upto \varepsilon_S(a_l)$.
Using going-down for the ring map $\varepsilon_S$ we get:
\[ \height_{\widehat{S}_{\qprime}}((I\widehat{S})_{\qprime}) \geq \height_{\widehat{S}}(I\widehat{S}) \geq \height_S(\varepsilon_S^{-1}(I\widehat{S})) \geq \height_S(I) = \height_{S_{\pprime}}(I_{\pprime}),\]
where the last equality holds since $I$ itself is a prime ideal (with $I \subseteq \pprime$). It follows that $(I\widehat{S})_{\qprime}$ is a complete intersection ideal.
\end{proof}

One case in which the assumption for an $n$-dimensional affine algebra $A$ to be a complete intersection in codimension $2 + \cid(A)$ (as in Proposition~\ref{PropositionAlgebraWithCIImpliesSpecOfComplIsSimplConnecInCodim2}) is certainly satisfied is when $A$ has isolated singularities (i.e., is regular in codimension $n-1$) and has $\cid(A) \leq n - 3$. Cutkosky \cite{cutkosky1995purity} gives various examples where this is the case (and where $A$ is not a complete intersection). One of them is the following.

\begin{example}\label{ExampleCIinCodimButNotCI}
The affine algebra
\[ A = K[x_1 \upto x_6] / I \text{ with } I = (x_1x_5 - x_2x_4,\, x_1x_6 - x_3x_4,\, x_2x_6-x_3x_5) \]
has dimension 4 and $\cid(A) = 1$. Using the Jacobian criterion (see \cite[Theorem 13.10]{kemper2010course}) we see that the origin in $K^6$ is the only singular point of the variety $X = \Var(I)$. So $A$ is regular in codimension $3$. In particular, $A$ is (not a complete intersection, but) a complete intersection in codimension $2 + \cid(A) = 3$.\hfill$\triangleleft$
\end{example}

Now again let $X$ be an affine $G$-variety. The affine algebra $A$ that appeared in Proposition~\ref{PropositionAlgebraWithCIImpliesSpecOfComplIsSimplConnecInCodim2} will be a separating subalgebra of the invariant ring later on. We will use a different characterization of separating algebras than Proposition~\ref{PropositionSeparatingRadicalCondition}.

\begin{prop}\label{PropositionSeparatingMeansInjectiveMorphism} A finitely generated subalgebra $A \subseteq K[X]^G$ is separating if and only if the induced morphism $\theta: \Spec(K[X]^G) \to \Spec(A)$ is injective.
\end{prop}

\begin{proof} Since $\Spec(A)$ and $\Spec(K[X]^G)$ are of finite type over the algebraically closed field $K$, it suffices to show that $\theta$ is injective on maximal ideals (see \cite[Theorem 2.2]{dufresne2009separating}). Every maximal ideal of $\Spec(K[X]^G)$ is of the form $\mmax^G$ with a maximal ideal $\mmax$ of $K[X]$, since  the quotient morphism $\pi: \Spec(K[X]) \to \Spec(K[X]^G)$ is surjective.

So take two maximal ideals $\mmax_x,\, \mmax_y$ of $K[X]$, which correspond to two points $x,\,y \in X$, and assume that $\mmax_x \cap A = \mmax_y \cap A$. Thus an invariant $g \in A$ vanishes at $x$ if and only if it vanishes at $y$. We get $g(x) = g(y)$ for all $g \in A$ from this. This gives $f(x) = f(y)$ for all $f \in K[X]^G$ by the separating property, hence $\mmax_x^G = \mmax_y^G$.
\end{proof}

In fact, under some mild additional assumptions on the separating algebra $A$ the map $\theta$ will not only be injective, but a \emph{universal homeomorphism} (i.e., every base change of $\theta$ is a homeo\-morphism). This also means that the property of being simply connected passes well between the spectrum of a separating algebra and the spectrum of the invariant ring.

\begin{lem}\label{LemmaThetaBetweenSeparatingAndInvariantRingIsUniversalHomeomorphism} Assume that $X$ is irreducible and that $A \subseteq K[X]^G$ is a finitely generated, separating algebra such that $K[X]^G$ is a finite $A$-module. Then the induced morphism  $\theta: \Spec(K[X]^G) \to \Spec(A)$ is a universal homeomorphism,
and the functor $\theta^{\ast}$ (as in Equation (\ref{EquationDefiningFunctorPullBack})) is an equivalence of categories.
\end{lem}

\begin{proof}
By Proposition~\ref{PropositionSeparatingMeansInjectiveMorphism}, $\theta$ is injective. In addition, $\theta$ is dominant. In \cite[Proposition 2.3.10]{derksen2002computational} it is shown that this implies that the extension of the fields of fractions 
$$\Quot(A) \subseteq \Quot(K[X]^G)$$ is finite and purely inseparable.
By the same argument, for all prime ideals $\pprime \subseteq K[X]^G$ the extension $\Quot(A / (\pprime \cap A)) \subseteq \Quot(K[X]^G / \pprime)$ is finite and purely inseparable. By \cite[Proposition 4.35]{gortz2010algebraic}, this means that $\theta$ is universally injective.

By assumption, $K[X]^G$ is a finite $A$-module, hence integral over $A$. So $\theta$ is finite, surjective and universally injective. Therefore $\theta$ is a universal homeomorphism (see \cite[Exercise 12.32]{gortz2010algebraic}) and the second claim follows by \cite[IX., Th\'eor\`eme 4.10]{grothendieck1964revetements}.
\end{proof}

\begin{remark}\label{RemarkAssumptionsInMainTheoremAboutCIareSatisfiedForGradedSepAlgebras} \begin{enumerate}[(a)]
\item Suppose that $K[X]$ is graded (e.g. if $X$ is the affine cone over a projective variety) and that the action of $G$ on $K[X]$ is degree-preserving (so that $K[X]^G$ is graded as well). Then every graded, finitely generated, separating subalgebra $A \subseteq  K[X]^G$ satisfies the assumption of Lemma~\ref{LemmaThetaBetweenSeparatingAndInvariantRingIsUniversalHomeomorphism}, i.e., $K[X]^G$ is a finite $A$-module. This is essentially shown in \cite[Theorem 2.3.12]{derksen2002computational}.
\item A situation where $K[X]$ is graded and the action is degree-preserving is the following: Suppose that $V$ is a linear representation of $G$. For every normal subgroup $N$ of $G$ we get an induced action of $G/N$ on $K[V]^N$ such that $K[V]^G = (K[V]^N)^{G/N}$. So the action of $G$ on $K[V]$ is split into two actions of smaller groups. This can be helpful in order to compute the invariant ring. But now $K[V]^N$ need not be a polynomial ring. However, it inherits the grading of $K[V]$ and the action of $G/N$ on $K[V]^N$ is degree-preserving.
\end{enumerate}
\end{remark}

As a last step before proving the main theorem of this section we collect the following facts about completions.

\begin{lem}\label{LemmaInvariantRingOfCompletion} Let $R := K[X]$ be the coordinate ring of the $G$-variety $X$, and let $\mmax \subseteq R$ be a maximal ideal fixed by $G$. Moreover, let $A \subseteq R^G $ be a finitely generated, separating subalgebra. Then the following hold: 
\begin{enumerate}[(a)] 
\item The $\mmax \cap A$-adic, the $\mmax^G$-adic, and the $\mmax$-adic filtrations on $R$ are all equivalent (i.e., they define the same topology on $R$).
\item The $\mmax \cap A$-adic, and the $\mmax^G$-adic filtrations on $R^G$ are equivalent.
\item The invariant ring of the $\mmax$-adic completion $\widehat{R}$ of $R$ (under the induced action of $G$ on $\widehat{R}$) is isomorphic to the $\mmax^G$-adic completion of the invariant ring:
\[ (\widehat{R})^G \cong \widehat{R^G}.\]
\end{enumerate}
\end{lem}

\begin{proof}(a)
Of course, we have $\sqrt{(\mmax \cap A) R} \subseteq \sqrt{\mmax^G R} \subseteq \m$. For part (a) we need show that these are actually equalities. For this let $\pprime \in \Spec(R)$ be any prime ideal containing $\mmax \cap A$. So in $A$ the inclusion $\mmax \cap A \subseteq \pprime \cap A$ holds.

It is $\mmax \cap A$ a maximal ideal in $A$ (see \cite[Prop. 1.2]{kemper2010course}). 
Hence we get $\mmax \cap A = \p \cap A$, and so in particular $\m^G \cap A = \p^G \cap A$. As $A$ is separating, this implies $\m^G = \p^G$ by Proposition~\ref{PropositionSeparatingMeansInjectiveMorphism}. Since the quotient map is a geometric quotient, this implies that there exists an element $\sigma \in G$ with $\sigma \p = \m$. But $\sigma^{-1} \m = \m$, hence $\p = \m$. \\
So $\m$ is the only prime ideal containing $(\mmax \cap A) R$, which gives $\sqrt{(\mmax \cap A) R} = \m$.

Part (b) follows with similar reasoning in $R^G$.

Part (c) is left to the reader. Also it is available in the author's dissertation \cite[Theorem 4.21]{reimers2016dissertation}.
\end{proof}

We come to our main result of this section. It extends Dufresne's result \cite[Theorem 1.3]{dufresne2009separating} to non-linear actions on normal and connected varieties (i.e., varieties whose coordinate ring is an integrally closed domain), and to separating algebras that are complete intersections in codimension $2 + \cid(A)$.

\begin{thm}\label{TheoremMainTheoremOnCompleteIntersectionAndBireflection}
Assume that the $G$-variety $X$ is normal and connected and that $X^G \neq \emptyset$. If there exists a finitely generated, separating algebra $A \subseteq K[X]^G$, such that $K[X]^G$ is a finite $A$-module, and that $A$ is a complete intersection in codimension $2 + \cid(A)$, then $G$ is generated by $2$-reflections. 
\end{thm}

\begin{proof}
Let $\mmax$ be a maximal ideal of $K[X]$ corresponding to a fixed point $x \in X^G$. By Lemma~\ref{LemmaInvariantRingOfCompletion} the $\mmax^G$-adic completion of $K[X]^G$ and the $\mmax$-adic completion of $K[X]$ are isomorphic to the $\mmax \cap A$-adic completion of $K[X]^G$ and $K[X]$, respectively. Furthermore, since $K[X]^G$ and $K[X]$ are finite $A$-modules, these completions are isomorphic to tensor products with $\widehat{A}$ (see  \cite[Theorem 8.7]{matsumura1989commutative}).

The inclusions $i: A \to K[X]^G$ and $j: K[X]^G \to K[X]$ induce ring homomorphisms between the completions of these rings. The situation is summarized in the following diagram:
\[
    \xymatrix{ 
    & \widehat{K[X]}^G \ar[d]^{\cong,\, \text{see Lemma 3.11(c)}} \\
    	\widehat{A} \ar[r]^{\widehat{i}}\ar[ur]^{=:\,\varphi} \ar[d]^{\cong} & \widehat{K[X]^G} \ar[r]^{\widehat{j}} \ar[d]^{\cong} & \widehat{K[X]} \ar[d]^{\cong} \\
    	\widehat{A} \otimes_A A \ar[r] & \widehat{A} \otimes_A K[X]^G \ar[r] & \widehat{A} \otimes_A K[X].
    }
\]

So the homomorphism $\varphi: \widehat{A} \to \widehat{K[X]}^G$ corresponds to the homomorphism 
\[ \id_{\widehat{A}} \otimes\, i \,\, :\quad \widehat{A} \otimes_A A \longrightarrow  \widehat{A} \otimes_A K[X]^G.
\]
Therefore, the scheme morphism $\omega := \Spec(\varphi)$ induced by $\varphi$ corresponds to a base change of the morphism $\theta = \Spec(i): \Spec(K[X]^G) \to \Spec(A)$. This map is a universal homeomorphism by Lemma~\ref{LemmaThetaBetweenSeparatingAndInvariantRingIsUniversalHomeomorphism}. Hence $\omega$ is a universal homeomorphism, too. 

With the complete intersection assumption on $A$, Proposition~\ref{PropositionAlgebraWithCIImpliesSpecOfComplIsSimplConnecInCodim2} shows that $\Spec(\widehat{A})$ is simply connected in codimension 2. 

Next we see that $\Spec(\widehat{K[X]}^G)$ is simply connected in codimension 2 as follows: Let $Z \subseteq \Spec(\widehat{K[X]}^G)$ be a closed subset of codimension $> 2$. Since $$\omega: \Spec(\widehat{K[X]}^G) \to \Spec(\widehat{A})$$ is finite, the set $\omega(Z)$ is closed and of codimension $> 2$ as well, hence $\Spec(\widehat{A}) \setminus \omega(Z)$ is simply connected. The restriction of $\omega$ gives a morphism $$\omega_0: \Spec(\widehat{K[X]}^G) \setminus Z \to \Spec(\widehat{A}) \setminus \omega(Z),$$ which is also a universal homeomorphism (since this property is ''local on the target'', see \cite[Corollary 4.33]{gortz2010algebraic}).
So as in the proof of Lemma~\ref{LemmaThetaBetweenSeparatingAndInvariantRingIsUniversalHomeomorphism}, it follows by \cite[IX., Th\'eor\`eme 4.10]{grothendieck1964revetements} that the functor $\omega_0^{\ast}$ is an equivalence of categories. By Remark~\ref{RemarkEquivalenceOfCategoriesAndSimplyConnected}, $\Spec(\widehat{K[X]}^G) \setminus Z$ is simply connected as well.

So we are in a position to apply Lemma~\ref{LemmaVinberg} to $\Spec(\widehat{K[X]})$. This space is irreducible by the normality assumption and a result of Zariski \cite{zariski1948irreducibilty}. Furthermore, its quotient is simply connected in codimension 2. So Lemma~\ref{LemmaVinberg} shows that $G$ is generated by elements $\sigma$ that are 2-reflections on $\Spec(\widehat{K[X]})$. 

But such an element $\sigma$ is a $2$-reflection on $X$ as well: For this we need to show that the ideal $J := \{ f - \sigma f \mid f \in K[X] \}$ in $K[X]$ has height $\leq 2$ (see Remark~\ref{RemarkKreflectionAndHeightOfCorrespondingIdeal}). By assumption on $\sigma$, the ideal $I := \{ f - \sigma f \mid f \in \widehat{K[X]} \}$ in $\widehat{K[X]}$ has height $\leq 2$. But going-down for the homomorphism $\varepsilon: K[X] \to \widehat{K[X]}$ then gives $\height_{K[X]}(\varepsilon^{-1}(I)) \leq 2$. And since $\varepsilon^{-1}(I)$ contains $J$, it follows that $\sigma$ is a $2$-reflection on $X$ as well.
\end{proof}

An example like Example~\ref{ExampleCIinCodimButNotCI} with an invariant ring $A = K[V]^G$ of a representation would certainly be a nice addendum to Theorem~\ref{TheoremMainTheoremOnCompleteIntersectionAndBireflection}. But for such invariant rings (at least in the non-modular case) being an isolated singularity is a rather strong condition (see \cite[Lemma 2.4]{stepanov2014gorenstein}). This leads to the following result as a corollary.

\begin{thm}\label{TheoremUsingIsolatedSingularityForAWeakCorollary}
Assume that $\characteristic(K) \nmid |G|$ and that $X = V$ is a non-trivial linear representation of $G$ with $n = \dim(V) \geq 3$. If $\cid(K[V]^G)~\leq~n-3$, then $G \setminus \{ \id \}$ contains an $(n-1)$-reflection.
\end{thm}

\begin{proof}
Assume that $G$ contains no $(n-1)$-reflections other than the identity element. By \cite[Lemma 2.4]{stepanov2014gorenstein}, these assumptions on $G$ and $V$ then imply that $K[V]^G$ is an isolated singularity. But $\cid(K[V]^G)~\leq~n-3$ now means that $K[V]^G$ is regular (hence complete intersection) in codimension $2 + \cid(K[V]^G)$. And by Theorem~\ref{TheoremMainTheoremOnCompleteIntersectionAndBireflection}, $G$ would be generated by 2-reflections, contradicting the assumption.
\end{proof}

For $K = \C$ Kac-Watanabe \cite[Theorem B]{kac1982finite} obtained a stronger result than Theorem~\ref{TheoremUsingIsolatedSingularityForAWeakCorollary}: If $\cid(\C[V]^G) = k$, then $G$ is generated by $(k+2)$-reflections. 

\begin{example} Let $\characteristic(K) \neq 3$ and let $\zeta \in K$ be a third root of unity. The invariant ring of 
\[ G = \langle \begin{pmatrix}1&0&0&0\\0&1&0&0\\0&0&\zeta&0\\0&0&0&\zeta\end{pmatrix}\rangle \leq \GL_4(K)\] is minimally generated by
\[ f_1 := x_1,\,\,\, f_2 := x_2,\,\,\,f_3 := x_3^3,\,\,\,f_4 := x_3^2x_4,\,\,\,f_5 := x_3x_4^2,\,\,\,f_6 := x_4^3.\] 
Between these 6 generators there are 3 relations:
\[f_3f_6 - f_4 f_5,\quad f_4^2 - f_3f_5,\quad f_5^2 - f_4f_6,\]
hence $\cidef(K[V]^G) = 1$. So the assumptions of Theorem \ref{TheoremUsingIsolatedSingularityForAWeakCorollary} are satisfied. The conclusion that $G$ contains a $3$-reflection other than the identity is not very strong here, as $G$ is indeed a $2$-reflection group. \end{example}

\section{Minimal Number of Separating Invariants}\label{SectionMinNumber}

This section extends the result of Dufresne and Jeffries \cite{dufresne2013separating} on $\gammasep$ to non-linear actions on varieties. The proof relies on Grothendieck's connectedness theorem and makes similar use of the separating variety $\Varsep$ as the results in Section \ref{SectionPolySepAlgebras}. 

Grothendieck's connectedness theorem is usually formulated with the notion of \emph{connectedness in dimension $d$}:
For an integer $d$, a Noetherian topological space $Y$ is called connected in dimension $d$
if for all closed subsets $Z \subseteq Y$ with  $\dim(Z) < d$ the space $Y \setminus Z$ is connected.

\begin{remark}\label{RemarkAboutConnectedInDimension}
If $R$ is the coordinate ring of an equidimensional affine variety and $Y = \Spec(R)$, then for all closed subsets $Z \subseteq Y$ the formula  
\begin{equation}\label{DimensionCodimensionFormula} \codim_Y(Z) + \dim(Z) = \dim(Y) \end{equation}
holds (see \cite[Corollary 8.23]{kemper2010course}). Hence $Y$ is connected in dimension $d$ if and only if $Y$ is connected in codimension $n-d$ where $n = \dim(Y)$. 

Now if $\widehat{R}$ is the completion of this coordinate ring $R$ at a maximal ideal, then $\widehat{R}$ is local, equidimensional and catenary (see the results of \cite[Section 31]{matsumura1989commutative}). Thus, formula (\ref{DimensionCodimensionFormula}) and the above conclusion about connectedness are also true for $Y = \Spec(\widehat{R})$.
\end{remark}

Next we state Grothendieck's connectedness theorem as taken from \cite[Theorem 19.2.12]{brodmann2013local}.

\begin{thm} \emph{\textbf{(Grothendieck)}}
\label{TheoremGrothendieckConnectedness} 
Let $R$ be a complete Noetherian local ring of dimension $n$ such that $\Spec(R)$ is connected in dimension $d$ with an integer $d < n$. Furthermore, let $\mmax$ be the maximal ideal of $R$ and $f_1 \upto f_r \in \mmax$. Then  $\Spec(R / (f_1 \upto f_r))$ is connected in dimension $d - r$.
\end{thm}

In order to use Grothendieck's connectedness theorem, we need to bring the connectedness properties back from $\Spec(\widehat{R})$ to $\Spec(R)$.

\begin{lem}\label{LemmaCompletionConnectedInCodimImpliesRingConnectedInCodim}
Let $(R,\,\mmax)$ be a Noetherian local ring and $(\widehat{R},\,\widehat{\mmax})$ its $\mmax$-adic completion.
If $\Spec(\widehat{R})$ is connected in codimension $k$, then $\Spec(R)$ is connected in codimension $k$, too.
\end{lem}

\begin{proof}
Since $\widehat{R}$ is a faithfully flat $R$-module, the morphism 
\[ \varphi: \Spec(\widehat{R}) \to \Spec(R),\quad \qprime \mapsto \qprime \cap R,\]
corresponding to the inclusion $\varepsilon: R \to \widehat{R}$, is surjective (see \cite[Theorem 7.3]{matsumura1989commutative}).

Now let $\pprime'$ and $\pprime''$  be minimal prime ideals of $R$. Because of the surjectivity of $\varphi$ there are prime ideals $\qprime'$ and $\qprime''$ of $\widehat{R}$ which are mapped to $\pprime'$ and $\pprime''$, respectively. Since $\widehat{R}$ is Noetherian, $\qprime'$ contains a minimal prime ideal of $\widehat{R}$, which has to be mapped to $\pprime'$, too, because of the minimality of $\pprime'$. Therefore, we can assume that $\qprime'$ and $\qprime''$ are minimal.
Now suppose that $\Spec(\widehat{R})$ is connected in codimension $k$. By (\ref{SecondVersionOfConnectednessInCodim}), this guarantees the existence of a finite sequence of minimal prime ideals
$\qprime_0 = \qprime', \, \qprime_1 \upto \qprime_r = \qprime''$ of  $\widehat{R}$ such that
\[ \height_{\widehat{R}}(\qprime_i + \qprime_{i+1}) = \codim_{\Spec(\widehat{R})}(\Var(\qprime_i) \cap \Var(\qprime_{i+1})) \leq k \quad \text{ (for }i= 0 \upto r-1). \]
With $\pprime_i := \qprime_i \cap R$ (for all $i$) we get:
\[ \height_R(\pprime_i + \pprime_{i+1}) \leq \height_R((\qprime_i + \qprime_{i+1}) \cap R) \leq \height_{\widehat{R}}(\qprime_i + \qprime_{i+1}) \leq k,\]
where the second inequality follows from going-down (see \cite[Theorem 8.8. \& Theorem 9.5]{matsumura1989commutative}).
So there is a finite sequence of irreducible closed subsets $Y_i := \Var(\pprime_i)$ of $\Spec(R)$ which leads from $Y_0 = \Var(\pprime')$ to $Y_r = \Var(\pprime'')$ in a way such that two subsequent subsets intersect in codimension $\leq k$. By (\ref{SecondVersionOfConnectednessInCodim}), $\Spec(R)$ is connected in codimension $k$.
\end{proof}

In the proof of Theorem~\ref{TheoremYsep=n+k-1=>k-reflections} we want to check the connectedness property only at maximal ideals. This works according to the following proposition.

\begin{prop}\label{PropositionLocalConnectednessImpliesGlobalConnectedness}
Let $R$ be a Noetherian Jacobson ring with connected spectrum $Y = \Spec(R)$, and assume that for all maximal ideals $\mmax$ of $R$ the spectrum of the localization $R_{\mmax}$ is connected in codimension $k$.
Then $Y$ is connected in codimension $k$, too.
\end{prop}

\begin{proof}
The proof is left to the reader. It is also available in the author's dissertation \cite[Proposition 2.7]{reimers2016dissertation}.
\end{proof}

Combining Grothendieck's connectedness theorem with Theorem~\ref{TheoremVsepConnectedInCodimIfAndOnlyIfGgenerated} leads to our main result of this section.

\begin{thm}\label{TheoremYsep=n+k-1=>k-reflections}
Assume that $X$ is normal and connected and that $G$ is generated by elements having a fixed point in $X$. If $\gammasep = n + k - 1$ (with $k \in \N$), then $G$ is generated by $k$-reflections.
\end{thm}

\begin{proof}
Let us write 
\[r := \gammasep = n + k - 1,\]
so there exists a separating subset $\{f_1 \upto f_r \} \subseteq K[X]^G$ of size $r$. Using the $\delta$-map from (\ref{EquationDefinitionOfDeltaMap}) we define
 \[g_i := \delta(f_i) \in K[X] \otimes_K K[X] \quad (\text{for } i = 1 \upto n).\] 
Moreover, let $J$ be the ideal in $R := K[X] \otimes_K K[X]$ generated by $g_1 \upto g_r$. Then
by Proposition~\ref{PropositionSeparatingRadicalCondition}, $\Spec(R/J)$ is homeomorphic to $\Spec(K[\Varsep])$ where $K[\Varsep]$ is the coordinate ring of the separating variety $\Varsep$. So according to Theorem~\ref{TheoremVsepConnectedInCodimIfAndOnlyIfGgenerated}, we need to show that $\Spec(R/J)$ is connected in codimension $k$, while we already know from Corollary~\ref{CorollaryVsepConnectedIfAndOnlyIf} that this space is connected. 

Take a point $(x,\,y) \in \Varsep$ and its corresponding maximal ideal $\mmax$ of $R$, and let $\widehat{R}$ be the $\mmax$-adic completion of $R$. By assumption, $X$ is normal, hence the product variety $X \times X$ is normal, too.
Normality of $X \times X$ implies that the spectrum of $\widehat{R}$ is irreducible (see \cite{zariski1948irreducibilty}). In particular, $\Spec(\widehat{R})$ is connected in dimension $d$ for all $d \leq 2n = \dim(\widehat{R})$. So we can apply the local version of Grothendieck's connectedness theorem, given in Theorem~\ref{TheoremGrothendieckConnectedness}, with $d = 2n-1$ to the complete Noetherian local ring $\widehat{R}$. It shows that the spectrum of $$\widehat{R} / J \widehat{R} = \widehat{R} / (g_1 \upto g_r)_{\widehat{R}}$$ is connected in dimension 
\[ (2n - 1) - r = 2n - 1 - (n+k-1) = n - k.\]
But by \cite[Theorem 8.11]{matsumura1989commutative}, $\widehat{R} / J \widehat{R}$ is the $\mmax$-adic completion of $R /J$. So we have 
\[ \dim(\widehat{R} / J \widehat{R}) = \height_{R/J}(\mmax/J) = \dim(R/J) = \dim(\Varsep) = n\]
(cf. Remark~\ref{RemarkVsepEqudimensionalToo}(a)). Therefore, the spectrum of $\widehat{R} / J \widehat{R}$ is connected in codimension $k$ by Remark~\ref{RemarkAboutConnectedInDimension}. 

Since $\widehat{R} / J \widehat{R}$ is also the completion of the local ring $(R/J)_{\mmax/J}$, Lemma~\ref{LemmaCompletionConnectedInCodimImpliesRingConnectedInCodim} shows that $\Spec((R/J)_{\mmax/J})$ is connected in codimension $k$.

As this holds for any maximal ideal $\mmax$ of $K[X] \otimes_K K[X]$ which corresponds to a point of $\Varsep$, i.e., any maximal ideal $\mmax / J$ of $R/J$, we conclude with Proposition~\ref{PropositionLocalConnectednessImpliesGlobalConnectedness} that 
$\Spec(K[\Varsep]) \cong \Spec(R/J)$ is connected in codimension $k$.
\end{proof}

Notice that the assumptions on $G$ and $X$ in Theorem~\ref{TheoremYsep=n+k-1=>k-reflections} are satisfied both for a linear action on a vector space $X=V$ (where every group element fixes the origin) and for the case of multiplicative invariants (where $X = \G_m^n$ and every group element fixes the point $(1 \upto 1)$). So Theorem~\ref{TheoremYsep=n+k-1=>k-reflections} is now also applicable for multiplicative invariant theory.

Furthermore, it is interesting to compare the case $k = 1$ of the above theorem with Theorem~\ref{TheoremYsep=nImpliesGeneratedByReflections}. In Theorem~\ref{TheoremYsep=n+k-1=>k-reflections} it was necessary to assume that $X$ is normal, while in Theorem~\ref{TheoremYsep=nImpliesGeneratedByReflections} it was necessary to assume that $X$ is Cohen-Macaulay. We can reuse Example~\ref{ExampleTwoAffinePlanesIntersectingInAPoint}, where $X$ was the union of two planes intersecting in a single point, to see that the assumption that $X$ is normal cannot be dropped from Theorem~\ref{TheoremYsep=n+k-1=>k-reflections}.

We finish with an example of a $G$-variety $X$ that is not an affine space and to which our main theorems apply.

\begin{example} Let $X$ be the following $3$-dimensional irreducible subvariety of $\C^4$:
\begin{equation}\label{DefinitionOfXinLastExample}X = \Var(\underbrace{x_1^3+x_2x_3 + x_3x_4 + x_4^2}_{=:h}).\end{equation}
Since $h(x_1,\,-x_2,\,-x_3,\,-x_4) = h(x_1,\,x_2,\,x_3,\,x_4)$, a cyclic group $G = \langle \sigma \rangle$ of order 2 acts on $X$ by
\begin{equation}\label{GroupActionInLastExample} \sigma \cdot (x_1,\,x_2,\,x_3,\,x_4) := (x_1,\,-x_2,\,-x_3,\,-x_4).\end{equation} 
The invariant ring of this $G$-variety is generated by 6 elements:
\begin{align}\label{InvariantRingInLastExample} \C[X]^G &= \C[x_1,\,\,x_2^2,\,\,x_3^2,\,\,x_4^2,\,\,x_2x_3,\,\,x_2x_4,\,\,x_3x_4] / (h)  \\
&= \C[\overline{x_1},\,\,\overline{x_2}^2,\,\,\overline{x_3}^2,\,\,\overline{x_2}\overline{x_3},\,\,\overline{x_2}\overline{x_4},\,\, \overline{x_3}\overline{x_4}].\notag\end{align}
The origin of $\C^4$ lies in $X$ and is fixed by this action, hence $X^{G} \neq \emptyset$. But there are no other fixed points in $X$: A point $(x_1,\,x_2,\,x_3,\,x_4) \in \C^4$ is fixed by the action defined in (\ref{GroupActionInLastExample}) if and only if $x_2=x_3=x_4 = 0$. But for $x \in X$ we see with (\ref{DefinitionOfXinLastExample}) that the vanishing of these three coordinates implies $x_1 = 0$, hence $X^{G} = \{ 0 \}$. In particular, $G$ is not generated by 2-reflections.

Since $X$ is a hypersurface, it is also Cohen-Macaulay. Thus, Theorem~\ref{TheoremYsep=nImpliesGeneratedByReflections} gives the lower bound $\gammasep \geq n + 1 = 4$.

Of course, we can do better now with Theorem~\ref{TheoremYsep=n+k-1=>k-reflections}. We need to check that $X$ is normal first. A quick calculation with the Jacobian criterion shows that the singular locus of $X$ is just $X^{\text{sing}} = \{ 0 \}$.
In particular, $\C[X]$ satisfies Serre's condition $(R_1)$. As $\C[X]$ is Cohen-Macaulay, it satisfies a fortiori the condition $(S_2)$. By Serre's criterion for normality (see \cite[Theorem 2.2.22]{bruns1993cohen}), we see that $\C[X]$ is normal. So Theorem~\ref{TheoremYsep=n+k-1=>k-reflections} implies that $\gammasep \geq 5$. In addition, Theorem~\ref{TheoremMainTheoremOnCompleteIntersectionAndBireflection} shows that no  separating subalgebra $A \subseteq \C[X]^G$ over which $\C[X]^G$ is integral is a complete intersection.

A generating set of size 6 was given in (\ref{InvariantRingInLastExample}). We claim that the following subset of size 5 is separating for the action on $X$:
\[ S = \{ \overline{x_1},\,\,\overline{x_2}^2,\,\,\overline{x_3}^2,\,\,\overline{x_2}\overline{x_4},\,\, \overline{x_3}\overline{x_4} \}.\] For the proof suppose that $u,\, v \in X$ satisfy $g(u) = g(v)$ for all $g \in S$. The invariant $f = \overline{x_2 x_3}$ is the only generator missing in $S$, so we just need to show that $f(u) = f(v)$. Write $u = (u_1,\,u_2,\,u_3,\,u_4)$. Then we have
\[ v = \begin{pmatrix}v_1\\v_2\\v_3\\v_4\end{pmatrix}\in \{ \begin{pmatrix}u_1\\\pm u_2 \\\pm u_3 \\ \lambda \end{pmatrix} \mid \lambda \in \C\}.\]
If $u_2 = 0$, then we get $f(u) = 0  = f(v)$ from this. So assume that $u_2 \neq 0$. With the invariant $\overline{x_2}\overline{x_4} \in S$ we get
 \[ u_2 u_4 = v_2 v_4  = \pm u_2 v_4, \]
hence $v_4 = \pm u_4$. The definition of $X$ in (\ref{DefinitionOfXinLastExample}) shows that $$f = \overline{x_2 x_3} = -\overline{x_1}^3 - \overline{x_3}\overline{x_4} - \overline{x_4}^2,$$
and for each of these three summands we now know that it takes the same value at $u$ and $v$, so $f(u) = f(v)$.

With Theorem~\ref{TheoremYsep=n+k-1=>k-reflections} we could really pin down $\gammasep$ to the value of $5$ in this example.\hfill$\triangleleft$
\end{example}


\bibliographystyle{myplain}
\bibliography{literature}

\end{document}